\documentclass{amsart}

\usepackage[margin=1in]{geometry} 
\usepackage{indentfirst}
\usepackage[english]{babel}
\usepackage{setspace}
\usepackage{thm-restate}
\usepackage{hyperref}
\hypersetup{
	colorlinks,
	linkcolor={MidnightBlue},
	citecolor={OliveGreen}
}

\usepackage{amsfonts,amsthm,latexsym,amsmath,amssymb,amscd,amsmath, mathrsfs, epsf,bm}
\usepackage[utf8]{inputenc}
\usepackage[dvipsnames]{xcolor}
\usepackage{tikz}
\usetikzlibrary{positioning, arrows.meta, 3d}
\usetikzlibrary{arrows,matrix}
\usepackage[nameinlink,noabbrev]{cleveref}
\usepackage{graphicx}
\usepackage[font = footnotesize]{caption}
\usepackage{enumitem}

\usepackage{booktabs}
\usepackage{comment}

\usepackage{nicematrix}
\usepackage[dvipsnames]{xcolor}
\definecolor{commentgreen}{RGB}{2,112,10}
\definecolor{eminence}{RGB}{108,48,130}

\definecolor{frenchplum}{RGB}{129,20,83}

\usepackage[normalem]{ulem}

\usepackage{listings}
\lstdefinelanguage{code}{
basicstyle=\small\ttfamily,
alsoletter=",
classoffset=1,
keywords={gb, eliminate, saturate, diff, degree, flatten, apply, tensor, product},
keywordstyle={\color{teal}},
classoffset=2,
morekeywords={from, to, list, terms, toList, entries, for, end, if, return},
keywordstyle={\color{blue}},
classoffset=3,
morekeywords={QQ},
keywordstyle={\color{frenchplum}},
classoffset=4,
morekeywords={ideal, matrix, gens},
keywordstyle={\color{teal}},
xleftmargin=1.5cm,
xrightmargin=1em,
columns=fullflexible,
keepspaces=true,
stepnumber=1,
numbers=none,
captionpos=b,
showspaces=false,
frame=none
}

\definecolor{weborange}{RGB}{255,165,0}
\definecolor{darkgray}{rgb}{.4,.4,.4}
\usepackage{mathtools}

\newtheorem{theorem}{Theorem}
\numberwithin{theorem}{section}
\newtheorem{proposition}[theorem]{Proposition}
\newtheorem{lemma}[theorem]{Lemma}
\newtheorem{corollary}[theorem]{Corollary}

\theoremstyle{definition}
\newtheorem{definition}[theorem]{Definition}
\newtheorem{remark}[theorem]{Remark}
\newtheorem{example}[theorem]{Example}
\newtheorem*{claim*}{\indent Claim}
\newtheorem*{problem*}{Problem}
\newtheorem{notation}[theorem]{Notation}

\newcommand{\Eig}{\mathrm{Eig}}
\newcommand{\W}{\mathcal{W}}
\newcommand{\PP}{\mathbb{P}}
\newcommand{\CC}{\mathbb{C}}
\newcommand{\RR}{\mathbb{R}}

\newcommand{\NN}{\mathbb{N}}

\newcommand{\rk}{\mathop{\rm rk}\nolimits}
\newcommand{\codim}{\mathop{\rm codim}\nolimits}

\newcommand{\sym}{\mathop{\rm Sym}\nolimits}

\title[On Waring Rank Jumps via critical rank-one approximations]{On Waring Rank Jumps \\ via critical rank-one approximations}
\author{Alessandro Oneto \and Pierpaola Santarsiero \and Ettore Teixeira Turatti}

\newcommand{\Addresses}{{% additional braces for segregating \footnotesize
  \bigskip
  \footnotesize

\textsc{Alessandro Oneto, 
Department of Mathematics (DIMA), University of Genoa, Via Dodecaneso 35, Genova, Italy}\par\nopagebreak
 \textit{E-mail address}:  \email{alessandro.oneto@unige.it}

\medskip

  \textsc{Pierpaola Santarsiero, Dipartimento di Ingegneria Industriale e Scienze Matematiche, Universit\`a Politecnica delle Marche, Via Brecce Bianche
I-60131 Ancona, Italy}\par\nopagebreak
  \textit{E-mail address}: \email{p.santarsiero@staff.univpm.it}

  \medskip

\textsc{Ettore Teixeira Turatti, Faculty of Mathematics, Informatics, and Mechanics, University of Warsaw, Banacha 2, 02-097 Warsaw, Poland} \par\nopagebreak
 \textit{E-mail address}: \email{e.teixeira-turatti@uw.edu.pl}
 }}


\begin{document}
\maketitle
\begin{abstract}
    We investigate whether eigenvectors, also known as critical rank-one approximations, of a symmetric tensor can be used to increase or decrease its Waring rank. First, we study the variety of degree-$d$ rank-$r$ forms which admit an eigenvector as part of a minimal Waring decomposition. In the case of binary forms, we show that this is of codimension-one in the $r$-th secant variety of the rational normal curve. On the other hand, we prove that {for any binary form of rank less than $(d+1)/2$ (subgeneric), any eigenvector increases the rank. Additionally, when the degree is odd, the same holds for generic forms of generic rank.} Our approach employs the strict relation between the apolar action and the Bombieri-Weyl product. 
 
\end{abstract}

\section{Introduction}
For a rank-$r$ matrix, it is well-known that singular pairs can be used to provide a decomposition of the matrix as sum of $r$ rank-one matrices. In the case of higher-order tensors a generalization of such a result is known to be not true in general.

Given a order-$d$ symmetric tensor $F \in \sym^d(\CC^{n+1})$, or degree-$d$ form, a \emph{Waring decomposition} of $F$ is an expression of $F$ as sum of powers of linear forms, i.e., 
\[
    F=\sum_{i=1}^r L_i^d,
\]
where $L_i\in \sym^1(\CC^{n+1})$. The minimal length of such a decomposition is called the \emph{Waring rank} of $F$. 

We say that a linear form $L\in\sym^1(\CC^{n+1})$ is an \emph{eigenvector} of $F\in\sym^d(\CC^{n+1})$ if $\nabla F(L)=\lambda L$, for some $\lambda\in \CC$.  Eigenvectors were introduced in \cite{Lim05,Qi07} as a generalization of singular pairs of matrices to higher order tensors from the perspective of optimization. In the literature, they are also referred to as \emph{critical rank-one approximations}, \emph{E-eigenvectors}, or \emph{singular vector tuples} for non-symmetric tensors.

From a geometric point of view, we consider the \emph{Veronese variety} $V_n^d$ parametrizing powers of linear forms inside the space of degree-$d$ forms. Given $F\in\sym^d(\CC^{n+1})$, a Waring decomposition corresponds to a set of points of the Veronese variety whose linear span contains $F$, while eigenvectors are the critical points of the squared Euclidean distance function from $F$ to the Veronese variety. 

Since computing minimal Waring decompositions is in general NP-hard \cite{HL13}, several approaches have been proposed to study algebraic and geometric properties that have to be satisfied by minimal Waring decompositions. Having in mind the ideal situation of matrices where the concepts of singular tuples and minimal additive decompositions are strictly related, it is very natural to try understanding the relation between eigenvectors of symmetric tensors and their minimal Waring decompositions. 

Interestingly, eigenvectors are typically \textit{not} part of a minimal Waring decomposition, even though sufficiently general symmetric tensors admit a (non-minimal) Waring decomposition made of eigenvectors \cite{DOT18}. In \cite{HT25}, the authors explored the question of when a form has a SVD-like decomposition, namely, when the process of replacing a symmetric tensor with its difference with an eigenvector eventually gets to zero: it was shown that only \textit{weakly-odeco forms} admit such property. The latter question was explored also in \cite{RHST25}. Towards the opposite direction, in \cite{SC10}, the authors showed that subtracting an eigenvector may actually increase the rank. This idea is interesting in view of the lack of knowledge of explicit high-rank forms, see \cite{buczynski2018locus}. Motivated by these previous works, we consider the following general problem.

\begin{problem*}
    Given $F \in \sym^d(\CC^{n+1})$, can its eigenvectors be used to decrease or increase its Waring rank? 
\end{problem*}

Our approach is based on the observation that the \textit{apolar action}, which allows to study Waring decompositions of $F$ via ideals of points contained in its annihilator with respect to derivation, and the \textit{Bombieri-Weyl product}, which is the scalar product used to define eigenvectors of $F$, are proportional.

\subsection*{Outline of the paper and main results}
In \Cref{section: prel} we recall preliminary notions, including the Bombieri-Weyl product and its connection with the apolar action. 

In \Cref{section: apolarity vs eigenscheme}, we investigate the relation between two geometric loci associated to $F$: the \textit{Waring locus} $\W(F)$, i.e., the locus of linear forms that appear in a minimal decomposition of $F$, and the \textit{eigenscheme} $\Eig(F)$, i.e., the locus of all eigenvectors of $F$. In other words, $\Eig(F) \cap \W(F)$ consists of the set of eigenvectors that can be employed to decrease the rank of $F$. In \Cref{lemma:apolar vs eig}, we characterize $\Eig(F) \cap \W(F)$ in terms of apolar ideals. We employ such characterization to give a complete description of $\Eig(F) \cap \W(F)$ for identifiable binary forms (\Cref{prop:binary rank 3} and \Cref{rmk: ann}) and monomials (\Cref{prop: osculante} and \Cref{prop: monomial case}). 

\Cref{section: critical W locus} is devoted to a geometric study of the \textit{critical Waring variety} $\mathrm{WE}^d_{n,r}$, that is the Zariski closure of the set of degree-$d$ rank-$r$ forms in $n+1$ variables such that $\Eig(F) \cap \W(F) \neq \emptyset$, see \Cref{def: critical waring locus}. In the case of binary forms, \Cref{proposition: dim WE rnc ,thm: irriducibilita'WE} show that $\mathrm{WE}_{1,r}^d$ is a codimension one irreducible subvariety of the $r$th secant variety of the rational normal curve. For forms in more variables, in \Cref{teo:dimWEn} we prove that $\mathrm{WE}_{n,r}^d$ has codimension at most $n$ in the $r$th secant variety of the Veronese variety.

\Cref{sec: grow} approaches the opposite direction of when eigenvectors allows to increase the Waring rank. In \Cref{cor:rankgroweig}, we show that for generic forms of subgeneric rank in any degree and for generic forms in odd degree, any eigenvector can be employed to increase the rank of $F$.

\subsection*{Acknowledgement}
The authors are grateful to Giorgio Ottaviani for interesting conversation on the topic. Santarsiero acknowledges the Lie-Størmer Center for supporting her visit to UiT in December 2024, from which this collaboration originated.

Oneto has been partially supported by the Italian Ministry of University and Research in the framework of the Call for Proposals for scrolling of final rankings of the PRIN 2022 call - Protocol no. 2022NBN7TL (“Applied Algebraic Geometry of Tensors”). Oneto is member of INdAM-GNSAGA. Santarsiero was supported by the European Union under NextGenerationEU. PRIN 2022, Prot. 2022E2Z4AK and PRIN 2022 SC-CUP: I53C24002240006.
Turatti was partially supported by Troms{\o} Research Foundation grant 17MATCR and the National Science Center, Poland, under the project “Tensor rank and its applications to signature tensors of paths”, 2023/51/D/ ST1/02363.

\section{Preliminary notions, apolar action and the Bombieri-Weyl product}\label{section: prel}
We work over $\CC$  and we use the standard identification between symmetric tensors in $\sym^d(\CC^{n+1})$ and degree $d$ homogeneous forms in $\CC[x_0,\dots,x_n]_{d}$. We start by introducing the main geometric objects used and we set the notation we will use throughout the whole paper. We refer to \cite{Hartshorne,landsberg2011tensors} as standard references for basics of algebraic geometry and tensor decomposition, respectively.

\subsection{Joins and secant varieties}

\begin{definition}
    Let $X,Y\subset \PP^N$ be projective varieties. The \emph{join} $\mathcal{J}(X,Y)$ of $X$ and $Y$ is the Zariski closure in $\PP^N$ of the union of lines passing through a point of $X$ and a point of $Y$:
    \[
        \mathcal{J}(X,Y)=\overline{\bigcup_{x\in X,y\in Y} \langle x,y \rangle }\subseteq \PP^N.
    \]
    Then, given $X_1,\dots,X_r\subset \PP^N$, we can iteratively define
    \[
        \mathcal{J}(X_1,\dots,X_r)=\mathcal{J}(\mathcal{J}(X_1,\ldots,X_{r-1}),X_r)=\overline{ \{q\in \langle x_1,\dots,x_r \rangle \,|\,  x_i\in X_i\} }.
    \]
\end{definition}

\begin{definition}
    Fix an integer $r\geq 2$. Given a variety $X\subset \PP^N$, set    
    \[
        \sigma_r^\circ(X)=\bigcup_{x_1,\dots,x_r\in X} \langle x_1,\dots,x_r\rangle.
    \]
    The \emph{$r$-th secant variety} $\sigma_r(X)$ of $X$ is the Zariski closure of $\sigma^\circ_r(X)$ in $\PP^N$. Equivalently, the $r$th secant variety is the join of $X$ with itself $r$ times, $\sigma_r(X) = \mathcal{J}(X,\ldots,X)$.
\end{definition}

\begin{remark}
    Recall that, if $X,Y$ are irreducible, then their join is irreducible. In particular, if $X$ is irreducible then all its secant varieties are irreducible. Moreover, by definition, secant varieties of $X$ cannot exceed the linear span of $X$. Therefore, from now on, we will always consider projective varieties that are irreducible and non-degenerate. 
\end{remark}

\subsection{Waring decompositions and Waring loci}
Additive decompositions as sums of powers are geometrically interpreted in terms of \emph{Veronese varieties} and their secant varieties.

\begin{definition}
    Fix integers $n,d$ with $n\geq 1$, $d\geq 2$. The \emph{$d$-th Veronese variety} $V^d_n$ is the image of the embedding $\PP(\sym^1(\CC^{n+1}))\rightarrow \PP(\sym^d(\CC^{n+1})),\ [v]\mapsto [v^d]$. 
    In other words, the Veronese variety $V_n^d$ parametrizes powers of linear forms. 
\end{definition}

\begin{definition}
    A \emph{Waring decomposition} of a form $F\in \sym^d(\CC^{n+1})$ is an expression 
    \[
        F=\sum_{i=1}^r L_i^d, \text{ where }L_i\in \sym^1(\CC^{n+1}).
    \]
    The \emph{Waring rank} of $F$ is the minimum length of a Waring decomposition, and we denote it by $\rk(F)$.
\end{definition}
Geometrically, saying that a form $F$ has rank $r$ is equivalent to say that $[F] \in \langle [L_1^d],\ldots,[L_r^d]\rangle.$ A priori, in the expression of $F$ as linear combination of the $L^d_i$'s there might be coefficients different than $1$: however, since we are working over $\CC$, we can always change representative of the projective point $[L_i^d]$ by bringing the coefficient under the exponent and, without loss of generality, we write $F = \sum_{i=1}^r L_i^d$. Then, it is immediate from the definitions that $\sigma_r^\circ(V^d_n)=\{F\in \PP(\sym^d(\CC^{n+1}))\, |\, \rk(F)\leq r \}$. 

In \cite{carlini2017waring}, it was introduced the locus of all points appearing in a minimal Waring decomposition.
\begin{definition}
Let $F\in \sym^d(\CC^{n+1})$. The \emph{Waring locus} $\W(F)$ of $F$ is
\begin{align*}
    \mathcal W(F)=\left\{[L]\in \PP(\sym^1(\CC^{n+1}))\mid\ \Large{\substack{F\in \langle L^d,L_2^d,\dots, L_r^d\rangle,\,  \rk(F)=r,\\
     [L_i]\in \PP(\sym^1(\CC^{n+1})),\ i=2,\dots,r}}\right\}. 
\end{align*}
Its complement, denoted by $\mathcal F(F)$, is called the \emph{forbidden locus} of $F$.
\end{definition}
Recall that $\W(F)$ could be neither a closed set nor an open set, hence the terminology ``locus''. The special case corresponds to when a form admits a unique decomposition, up to permutation.

\begin{definition}
A rank-$r$ form $F=\sum_{i=1}^rL_i^d\in \sym^d(\CC^{n+1})$ is called \emph{identifiable} if $\W(F)=\{[L_1],\dots,[L_r]\}$.    
\end{definition}

\subsection{Apolar action and Apolarity Lemma}
Apolarity theory offers a solid algebraic framework in which Waring decompositions of a given form $F$ can be studied in terms of ideals of points contained in the annihilator of $F$ with respect to derivation. For a complete overview, we refer to \cite{IarrobinoKanev}.

We identify $\sym^d(\CC^{n+1})=\CC[x_0,\dots,x_n]_d$ and $(\sym^d(\CC^{n+1}))^\ast=\CC[\partial_0,\dots,\partial_n]_d$. We consider the \emph{apolar action} via differentiation
\[
    G\circ F=G(\partial_0,\dots,\partial_n) F(x_0,\dots,x_n),\ F\in \sym^d(\CC^{n+1}),\ G\in \sym^d(\CC^{n+1})^*.
\]
Then $\CC[x_0,\dots,x_n]$ 
has the structure of a $\CC[\partial_0,\dots,\partial_n]$-module 
via differentiation.
\begin{definition}
    Let $F\in \sym^d(\CC^{n+1})$, the \emph{apolar ideal} of $F$ is $$
    \mathrm{Ann}(F)=\{G \in \CC[\partial_0,\dots,\partial_n]\mid G \circ F=0\}.
    $$
\end{definition}
The apolar action and the apolar ideal are one of the main tools to compute Waring decompositions due to the apolarity lemma, see for instance \cite[Lemma 1.15]{IarrobinoKanev}.

\begin{lemma}[Apolarity Lemma]
     Let $F\in \sym^d(\CC^{n+1})$.  Let $Z=\{P_1,\dots, P_r\}\subset \PP^n$ be a set of points with $P_i=[(p_{i_0}:\dots:p_{i_n})]$. Let $L_i=p_{i_0}x_0+\dots+p_{i_n}x_n$, for $i=1,\dots,r$, the corresponding linear forms. Then, $I_Z\subset \mathrm{Ann}(F)$ if and only if $F\in \langle L_1^d,\dots,L_r^d\rangle$.
\end{lemma}

With a little abuse of notation, for $F,G\in\CC[x_0,\ldots,x_n]$, we will write $G \in {\rm Ann}(F)$ when $G(\partial)\circ F = 0$ through the substitution $x_i \mapsto \partial_i$. 

\subsection{Bombieri-Weyl product and the eigenscheme}
Let $\langle \, \cdot \, , \, \cdot \,  \rangle_\RR:\RR^{n+1}\times \RR^{n+1}\to \RR$ be a non-degenerate bilinear form in $\RR^{n+1}$, and consider its extension to the complex numbers as a function $\langle \, \cdot\, ,\, \cdot \,  \rangle_\CC:\CC^{n+1}\times\CC^{n+1}\to \CC$. Notice that this means that there exist vectors $v\in\CC^{n+1}\setminus \{ 0\}$ such that $\langle v,v\rangle_\CC=0$, those are called \emph{isotropic vectors}. The \emph{Bombieri-Weyl product} $\langle \, \cdot \,,\, \cdot \,\rangle $ in $\sym^d(\CC^{n+1})$ is defined on rank-one forms as
\[
    \langle u^d,v^d\rangle=(\langle u,v\rangle_\CC)^d
\]
and extended by linearity to $\sym^d(\CC^{n+1})$ as every polynomial can be seen as a linear combination of $d$th-powers of linear forms. 

The next result is a standard fact connecting the Bombieri-Weyl product for which variables are orthonormal and the apolarity action. For the sake of completeness, we formalise it in the next lemma as it will be central in connecting Waring decompositions and the eigenscheme. 

\begin{notation}
    We write $x^\alpha := \prod_{i=0}^n x_i^{\alpha_i} \in\sym^{|\alpha|}(\CC^{n+1})$, where $\alpha=(\alpha_0,\dots,\alpha_n)$ and $|\alpha|=\sum_{i=0}^n\alpha_i$.
\end{notation}

\begin{lemma}
    Let $F=\sum_{|\alpha|=d}F_\alpha x^\alpha,G=\sum_{|\alpha|=d}G_\alpha x^\alpha\in \sym^d(\CC^{n+1})$. Then, $$
    G(\partial)\circ F=d!\langle F,G\rangle,
    $$ or, in other words, the Bombieri-Weyl product for which variables are orthonormal is, up to a scalar, the apolar action restricted to $\sym^d(\CC^{n+1})\times (\sym^d(\CC^{n+1}))^\ast\to \CC$.
\end{lemma}

\begin{proof} On the one hand, one easily computes
$$\langle F,G \rangle=\sum_{|\alpha|=d} \binom{d}{\alpha}^{-1}F_\alpha G_\alpha,$$
see for instance \cite[Equation 1.4]{Reznick}. On the other hand 
\begin{equation*}
G(\partial) \circ F=\left(\sum_{|\alpha|=d} G_\alpha \partial^\alpha\right)\left(\sum_{|\alpha|=d} F_\alpha x^\alpha\right)=\sum_{|\alpha|=d} \alpha!F_\alpha G_\alpha.\qedhere
\end{equation*}   
\end{proof}

\begin{definition}
    Let $F\in \sym^d(\CC^{n+1})$, we say that $L$ is a \emph{eigenvector} of $F$ with \emph{singular value} $\lambda\in\CC$ if $$0 = \langle F-\lambda L^d,L^{d-1}M\rangle= L^{d-1}M\circ(F-\lambda L^d),\ \text{for every $M\in\sym^1(\CC^{n+1})$}.$$  
\end{definition}

\begin{remark}\label{rmk:tangent}
    If we recall that the tangent space to the Veronese variety $V_n^d$ at a point $[L^d]$ is $T_{[L^d]}V_n^d = \{[L^{d-1}M] \mid M \in \sym^d(\CC^{n+1})\}$, then we can say that $L$ is an eigenvector of $F$ if the line $\langle [F],[L^d] \rangle$ is orthogonal to the tangent space to the Veronese variety at $[L^d]$ with respect to the Bombieri-Weyl product. 
\end{remark}

Notice that $L$ being an eigenvector of $F$ is equivalent to $\nabla(F(L))=\lambda' L$, where by $F(L)$ we denote the evaluation of $F$ in the point corresponding to $L$. Indeed, we have $
 L^{d-1}M\circ(F-\lambda L^d)=0
$ for all $M\in \CC^{n+1}$, if and only if $ L^{d-1}\circ(F-\lambda L^d)=0$, which is equivalent to $\nabla F(L)=\nabla L^d(L)=\lambda' L$, where $\lambda'=\frac{\lambda}{d} L^{d-1}(\partial)\circ L^{d-1}.$

\begin{definition}\label{def:eigenscheme}
    The \emph{eigenscheme} of $F\in \sym^d(\CC^{n+1})$, denoted by $\mathrm{Eig}(F)$, is the scheme in $\PP(\sym^1(\CC^{n+1}))$ defined by 
    \begin{align}\label{eq: defining ideal eigenscheme}
    \mathrm{minors}_{2\times2}\begin{bmatrix}
        \nabla F(x)\\
    x
    \end{bmatrix}=0.
    \end{align}
    It consists of all linear forms that are eigenvectors of $F$.
\end{definition}

\begin{remark}
We remark that if $L$ is isotropic and $L$ is an eigenvector of $F$, then its singular value is not unique: indeed, in such case we have $L^{d-1}\circ (F-\lambda L^{d-1})= L^{d-1}\circ F=0$, since $L\circ L=0$, thus the condition is independent of $\lambda\in\CC$. On the other hand, if $L$ is not isotropic, then its singular value is uniquely determined, see \cite[Proposition 2.6]{DOT18}.  
\end{remark}

It is also  worth noticing that, unlike matrices, the value of $\lambda$ depends on the scaling of $L$. Indeed, if $L^d$ is an eigenvector of $F$ with singular value $\lambda$, then $\alpha L^d$ has singular value $\lambda / \alpha$ {for any $\alpha\neq0$}.  To mitigate this issue, when $L$ is not isotropic and the singular value is the object of study, it is often assumed that $\langle L,L\rangle_\CC=1$ to avoid ambiguity. In our case, as we are interested in the eigenscheme, $L$ is not assumed to have norm one.

\medskip
Before moving to our results, we recall a standard working assumption.

\begin{remark}    
Let $F\in \sym^d(\CC^{n+1})$. We say that $F$ has $s$ \emph{essential variables} if there exists a vector subspace $V\subseteq \sym^1(\CC^{n+1})$ with $\dim V=s$ such that $F\in \sym^d(V)\subseteq \sym^d (\CC^{n+1})$. It is well-known that minimal Waring decompositions of $F$ involve only linear forms in its essential variables, in other words, we have that $\W(F)\subseteq \PP(V)$, e.g., see \cite[Remark 2.4]{carlini2017waring}. On the other hand, if $s<n+1$, then $\Eig(F)\not\subset \PP(V)$ since $\PP(V^\perp)\subset \Eig(F)$. 
However, the critical points of $\PP(V^\perp)$ are generally not so interesting as is evident for instance in the case $d=2$ where the singular pairs associated with zero singular values 
of a non-maximal rank matrix provide no information on lower rank approximations of the matrix. 
\end{remark}

\section{Apolar Ideal and Eigenscheme}\label{section: apolarity vs eigenscheme}
We first study when eigenvectors can be employed to decrease the rank. In other words, given a form $F$, we want to understand the intersection $\W(F) \cap \Eig(F)$. In this section, we establish the basic connections between the Waring locus and the eigenscheme of a form. 

\begin{notation}
 Given a vector subspace $W\subset \sym^d(\CC^{n+1})$, we write $\langle F,W \rangle=0$ to mean $\langle F,G\rangle=0 ~\forall G\in W$.
\end{notation}

\begin{lemma}\label{lemma:apolar vs eig}
Let $n\geq 1$, $d\geq 2$, and let $F\in \sym^d(\CC^{n+1})$. Then
 $[L]\in \Eig(F)$ if and only if there exists $\mu \in \CC$ such that $L^{d-1}\in \mathrm{Ann}(F-\mu L^d)$. In particular, $\mu$ is the singular value of $L$. 

\end{lemma}
\begin{proof}
Suppose $L$ is an eigenvector of $F$. Then, as commented in \Cref{rmk:tangent}, this is equivalent to the existence of $\mu\in \CC$ such that $\langle F-\mu L^d,T_{[L^d]}V_n^d\rangle=0$, that is, $\langle F-\mu L^d,ML^{d-1}\rangle=0$ for all $M\in \sym^1(\CC^{n+1})$. Equivalently, $L^{d-1}\circ (F-\mu L^d)=0$, namely
$L^{d-1}\in \mathrm{Ann}(F-\mu L^d)$.
\end{proof}

\begin{remark}\label{rmk:singularvalue=coefficient}
 Notice that, if $[L]\in \W(F)$, letting $F=\alpha L^d+\sum_{i=2}^rL_i^d$ with $r = \rk(F)$, then by \Cref{lemma:apolar vs eig} $L$ is also an eigenvector if and only if $L^{d-1}\in\mathrm{Ann}((\alpha-\mu)L^d+\sum_{i=2}^rL_i^d)$. In particular, it has singular value $\mu=\alpha$ if and only if $L^{d-1}\in \mathrm{Ann}(\sum_{i=2}^rL_i^d)$, i.e., $L$ is an eigenvector of $\sum_{i=2}^rL_i^d$ with singular value zero. 

\end{remark}

The case of forms of rank $r$ in $r$ essential variables is special as the linear forms appearing in the Waring decomposition are linearly independent. In this case, the fact that a summand of a minimal Waring decomposition is also an eigenvector is equivalent to the orthogonality with the other summands.

\begin{proposition}\label{prop:essential_var}
    Let $F$ with $r$ essential variables and $\rk(F)=r$, then $[L] \in \mathcal W(F)\cap \mathrm{Eig}(F)$ if and only if $F = \alpha L^d + \sum_{i=2}^r \beta_iL_i^d$ with $\langle L,L_i\rangle_\CC=0$, for all $i = 2,\ldots,r$, and either $\alpha$ is the singular value of $L$ as eigenvector of $F$ or $L$ is isotropic.
\end{proposition}

\begin{proof}
    By definition, $[L] \in \W(F)$ if and only if $F \in \langle L^d,L_2^d,\ldots,L_r^d\rangle$, say $F = \alpha L^d + \sum_{i=2}^r \beta_iL_i^d$. At the same time, $[L] \in \Eig(F)$ if and only if there exists $\mu \in \CC$ such that
    \begin{equation}\label{eq:3_1}
        L^{d-1}\circ\left((\alpha-\mu) L^d+\sum_{i=2}^r\beta_iL_i^d\right)=d!(\alpha-\mu)\langle L,L\rangle_\CC^{d-1}L+d!\sum_{i=2}^r\beta_i\langle L,L_i\rangle_\CC^{d-1}L_i=0.
    \end{equation}
    Since $F$ has $r$ essential variables, $\{L,L_2,\ldots,L_r\}$ are linearly independent. Thus, \cref{eq:3_1} is equivalent to have that $\langle L,L_i \rangle_\CC = 0$, for all $i = 2,\ldots,r$, and either $\langle L,L \rangle_\CC = 0$, i.e., $L$ is isotropic, or $\alpha = \mu$. {Recall that, if $L$ is isotropic, any $\mu \in \CC$ is a valid eigenvalue for $L$, see \Cref{rmk:singularvalue=coefficient}.}
\end{proof}
\begin{example}
    \begin{enumerate}
        \item Let $F=(x_0+ix_1)^d+x_2^d+(x_2+x_3)^d\in \sym^d\CC^{4}$. Notice $F$ has three essential variables and rank three. Moreover, $x_2$ and $(x_2+x_3)$ are not eigenvectors, and $x_0+ix_1$ is an isotropic eigenvector, in particular any $\mu\in\CC$ is a valid singular value.
        \item $F=\alpha x_0^d+x_1^d+(x_1+x_2)^d\in \sym^d\CC^{3}$ has three essential variables and rank three. Moreover, $x_1$ and $(x_1+x_2)$ are not eigenvectors, and $x_0$ is an eigenvector with singular value $\alpha$.

    \end{enumerate}
    \end{example}

\subsection{Binary forms}
The apolar ideal of any binary form $F \in \sym^d(\CC^2)$ is generated by two forms, ${\rm Ann}(F) = (G_1,G_2)$ with $\deg(G_1)+\deg(G_2)=d+2$, see \cite[Theorem 1.44]{IarrobinoKanev}. By Apolarity Lemma, it follows that, assuming $\deg(G_1) \leq \deg(G_2)$, then $\rk(F) = \deg(G_1)$ if $G_1$ is squarefree and $\rk(F) = \deg(G_2)$ otherwise. This is a reformulation of a classical result by Sylvester on binary forms. This description of the apolar ideal allows us to obtain a more precise characterization of when a linear form is both in the Waring locus and in the eigenscheme of a binary form.

\begin{notation}
    Let $[L] \in\PP \sym^1\CC^2$, we denote by $[L^\perp]\in\PP\sym^1\CC^2$ the unique linear form (up to scalar) such that $L^\perp\circ L = \langle L^\perp,L \rangle_\CC=0$.
    Notice that if $L$ is isotropic, which for binary forms means $[L]\in\{[x_0+ix_1],[x_0-ix_1]\}$, then $[L^\perp]=[L]$.
\end{notation}

\begin{proposition}\label{prop:binary rank 3}
    Let $F \in \sym^d\CC^2$ with $\rk(F)=3$. If $[L] \in \W(F) \cap \Eig(F)$, $L$ appears with coefficient $\alpha$ in a minimal Waring decomposition of $F$ and $\mu$ is the singular value of $L$ as eigenvector of $F$, then, $\mu \neq \alpha$ and $L$ is not isotropic. Moreover, if $d \geq 4$, $[L] \in \W(F) \cap \Eig(F)$ implies that $\mathrm{Ann}(F-\mu L^d)=(L^\perp L_2^\perp L_3^\perp,L^{d-1})$.
\end{proposition}
\begin{proof}
    Let $[L]\in\W(F) \cap \Eig(F)$, then, $F = \alpha L^d + \beta_2L_2^d + \beta_3L_3^d$. Assume, by contradiction, that $\mu = \alpha$. In this case, $F - \mu L^d = \beta_2L_2^d + \beta_3L_3^d$ has rank equal to $2$ and ${\rm Ann}(F-\mu L^d) = (L_2^\perp L_3^\perp,H)$ with $\deg(H) = d$. By \Cref{lemma:apolar vs eig}, $L^{d-1} \in {\rm Ann}(F-\mu L^d)$, thus $L_2^\perp L_3^\perp \mid L^{d-1}$: a contradiction, since $L_2$ and $L_3$ are not proportional. 
    
    Note that $L$ cannot be isotropic. Otherwise, \cref{eq:3_1} simplifies to $\beta_2\langle L,L_2 \rangle_\CC^{d-1}L_2 + \beta_3\langle L,L_3 \rangle_\CC^{d-1}L_3 = 0$. Since $L_2$ and $L_3$ are not proportional and $\beta_2\beta_3 \neq 0$, it follows $\langle L,L_2 \rangle_\CC = \langle L,L_3 \rangle_\CC = 0$: a contradiction, since we are working with binary forms. 

    From the first part of the proof, if $[L] \in \W(F) \cap \Eig(F)$ and $\mu$ is the singular value of $L$, then $F-\mu L^d$ has rank equal to $3$: thus, since $d \geq 4$, a generator of ${\rm Ann}(F-\mu L^d)$ of smallest degree must be degree-$3$, square-free and can be chosen to be divisible by $L^\perp$, say $L^\perp L_2^\perp L_3^\perp$ and $F \in \langle L^d,L_2^d,L_3^d\rangle$. At the same time, by \Cref{lemma:apolar vs eig}, $L^{d-1} \in {\rm Ann}(F-\mu L^d)$. Since $L^\perp L_2^\perp L_3^\perp$ is square-free, it cannot divide $L^{d-1}$ and, since $\deg(L^\perp L_2^\perp L_3^\perp) + \deg(L^{d-1}) = d+2$, $L^{d-1}$ can be taken as generator of largest degree; thus, ${\rm Ann}(F-\mu L^d) = (L^\perp L_2^\perp L_3^\perp,L^{d-1})$. 
\end{proof}

\begin{remark}
    Notice that \cite{HT25} asks the question of when an eigenvector can be subtracted with a coefficient equal to its singular value and the rank reduces by one. In \Cref{prop:essential_var}, we showed that if the rank is equal to the number of essential variables then it can be done as long as we have a summand which is orthogonal to all the others. In \Cref{prop:binary rank 3}, we show that already for rank-$3$ binary forms, even if $[L]\in\mathcal W(F)\cap \mathrm{Eig}(F)$, $L$ must appear with a coefficient $\alpha$ in the minimal Waring decomposition of $F$ different from its singular value $\mu$.
    Thus, $F-\mu L^d$ still has rank three.
\end{remark}

\begin{example}
    It is not hard to see that there exists rank three binary forms $F\in\sym^d(\CC^2)$ with $\mathcal W(F)\cap\mathrm{Eig}(F)\neq \emptyset$. Indeed, let $F=x^d+y^d+(x+y)^d$, then, for all $d\geq3$, $[x+y]\in \mathcal W(F)\cap\Eig(F)$, while $[x],[y]\not\in\mathrm{Eig}(F)$. It is possible to check that the singular value of $x+y$ is $1+2^{1-d} \neq 1$.
\end{example}

\begin{remark}\label{rmk: ann}
    Let $[L] \in \W(F) \cap \Eig(F)$ with $F \in \langle L^d,L_2^d,\ldots,L_r^d \rangle$. The case $\rk(F) =r=3$ and $d \geq4$ is special because, as we have seen in  \Cref{prop:binary rank 3}, $F-\mu L^d$ is guaranteed to be still rank-$3$ and the degree of the largest-degree generator of $\mathrm{Ann}(F-\mu L^d)$ is $d-1$; therefore, $L^{d-1}$ could be taken as such minimal generator. If $4 \leq r \leq \frac{d+2}{2}$, $F - \mu L^d$ could be either of rank $r-1$ or $r$, depending if $\mu = \alpha$ or $\mu \neq \alpha$, respectively. In the first case ($\mu = \alpha$), ${\rm Ann}(F-\mu L^d) = (L_2^\perp\cdots L_r^\perp, H)$ with $\deg(H) = d+3-r$; in the second one ($\mu \neq \alpha$), ${\rm Ann}(F-\mu L^d) = (L^\perp L_2^\perp\cdots L_r^\perp, H)$ with $\deg(H) = d+2-r$. Thus, unless $d = 4$ and $\mu = \alpha$, $L^{d-1} \in {\rm Ann}(F-\mu L^d)$ but it cannot be taken as minimal generator.
 \end{remark}

\begin{remark}\label{rmk:algo_binary}
    In the case of binary forms, the eigenscheme of $F \in \sym^d(\CC^2)$ is a zero-dimensional of degree-$d$ in $\PP^1$, see \Cref{def:eigenscheme}. In other words, up to multiplicity, it is supported on finitely many points. Thus, from \cite[Theorem 3.5]{carlini2017waring}, we have a direct and finite way to check whether $\W(F) \cap \Eig(F)$ is not empty:
    \begin{itemize}
        \item compute the support of $\Eig(F)$;
        \item compute ${\rm Ann}(F) = (G_1,G_2)$ with $\deg(G_1) \leq \deg(G_2)$;
        \item if $G_1$ is square-free then $\W(F)$ is the set of roots of $G_1$, then we have to check if any of these is in the support of $\Eig(F)$; or,
        \item if $G_1$ is not square-free and $\deg(G_2) > \deg(G_1)$, then $\mathcal{F}(F)$ is the set of roots of $G_1$, then we have to check that not all the support of $\Eig(F)$ is contained in it; or,
        \item if $G_1$ is not square-free and $\deg(G_1) = \deg(G_2)$, the proof of \cite[Theorem 3.5]{carlini2017waring} gives a recipe to compute the forbidden locus of $F$ and, again, we have to check that not all the support of $\Eig(F)$ is contained in it.
    \end{itemize}
\end{remark}
We explicitly follow \Cref{rmk:algo_binary} for the family of binary monomials $L^{d-j}M^j$ for $1 \leq j \leq d/2$, with $L,M$ not proportional, whose rank is $d-j+1$, see \cite{CCG12}. Namely, for the elements lying on the $j$th osculating variety of the Rational Normal Curve $V^d_1$, but not on $V^d_1$. First, we compute the eigenscheme.

\begin{proposition}\label{prop: osculante}
    Let $F=L^{d-j}M^j$ for $1\leq j\leq d/2$ for linearly independent forms $L=a_0x_0+a_1x_1$, $M=b_0x_0+b_1x_1$. Denote $\widetilde{L}=-a_1x_0+a_0x_1$ and $\widetilde{M}=-b_1x_0+b_0x_1$. Let $Q_1,Q_2$ be solutions of the quadric 
    \[
       D= \det \begin{bmatrix}
            (d-j)M & jL \\
            -\widetilde{M} & \widetilde{L}
        \end{bmatrix}=0.
    \]
    Then, if $j=1$, $\Eig(F)=\{[L^\perp], [Q_1],[Q_2]\}$, otherwise $\Eig(F)=\{[L^\perp],[M^\perp], [Q_1],[Q_2]\}$.  
\end{proposition}
\begin{proof}
The result follows directly by computing $\Eig(F)$ via \cref{eq: defining ideal eigenscheme}:
  \begin{equation*}
  \det\begin{bmatrix}
        \nabla F(x)\\
    x
    \end{bmatrix}=L^{d-j-1}M^{j-1}\left( (d-j)M\widetilde L+jL\widetilde M \right) = L^{d-j-1}M^{j-1} \cdot D       \qedhere
\end{equation*}
\end{proof}

Note that in the latter proposition, $[\widetilde{L}] = [L^\perp]$, and similarly for $M$.

\begin{corollary}\label{cor: osculante}
    Let $F=L^{d-j}M^j$ for $1\leq j\leq d/2$ and $D, Q_1,Q_2$ be as in \Cref{prop: osculante}, and denote $S=\{[L],[L^\perp],[M],[M^\perp]\}$. 
    Then:
    \begin{enumerate}
    \item if both $L$ and $M$ are isotropic, then $\W(F) \cap \Eig(F) = \{[M]\}$; otherwise,
    \item if $L$ and $M$ are not both isotropic, then 
        \begin{enumerate}
            \item if $j=1$,  
            \[
                \mathcal W(F)\cap \Eig(F)=     \begin{cases}
                   \{[Q_2]\}   & \text{ if }L \text{ is isotropic, where $([Q_1],[Q_2])=([L],[Q_2])$ and $[Q_2]\notin S$},\\
                   \{[L^\perp],[M],[Q_2]\}   & \text{ if }M \text{ is isotropic, where $([Q_1],[Q_2])=([M],[Q_2])$ and $[Q_2]\notin S$},\\
                       \{[L^\perp],[Q_1],[Q_2]\} & \text{ otherwise, with $[Q_i]\notin S$}.
                    \end{cases}
            \]
            \item if $j > 1$, 
            \[
                \mathcal W(F)\cap \Eig(F)=     \begin{cases}
                   \{[M^\perp], [Q_2]\}   & \text{ if }L \text{ is isotropic, where $([Q_1],[Q_2])=\{[L],[Q_2]\}$ and $[Q_2]\notin S$}, \\
                   \{[L^\perp],[M],[Q_2]\}   & \text{ if }M \text{ is isotropic, where $([Q_1,Q_2])=\{[M],[Q_2]\}$ and $[Q_2]\notin S$},\\
                       \{[L^\perp],[M^\perp], [Q_1],[Q_2]\} & \text{ otherwise,  with $[Q_i]\notin S$}.
                    \end{cases}
            \]
        \end{enumerate}
        \end{enumerate}

\end{corollary}
\begin{proof}
 The apolar ideal to $F$ is $\mathrm{Ann}(F)=((L^\perp)^{j+1},H)$, for some $H\in\sym^{d-j+1}(\CC^2)$, so that the generator of minimal degree is $(L^\perp)^{j+1}$. By \cite[Theorem 3.5]{carlini2017waring}, the forbidden locus of $F$ is $\mathcal{F}(F)=\{ [L]\}$, so its Waring locus is $\W(F)=\PP^1\setminus\{ [L]\}$. Moreover, it is easy to see that $[Q_i]\not\in\{[L],[L^\perp],[M],[M^\perp]\}$, unless $[L]$ is isotropic, in which case $[L]=[L^\perp]$ is a root of $D$, or, if $[M]$ is isotropic, in which case $[M]=[M^\perp]$ is a root of $D$. Indeed: if $L$ is isotropic, then the evaluation of the quadric $D$ at $L$ is equal to $0$ because $L \circ L^\perp = L \circ L = 0$; while, if $L$ is not isotropic, then the evaluation of $D$ at $L$ is equal to $j (L\circ L) (L \circ \widetilde{M})$ which is different than zero because $L$ and $M$ are not proportional.
\end{proof}

For binary forms of rank greater than $(d+1)/2$ that are not monomials we can easily give a sufficient condition to understand whether the intersection of the eigenscheme and the Waring locus is nonempty.
\begin{proposition}
    Let $F\in \sym^d(\CC^2)$ with $\rk(F)>(d+1)/2$. If $\Eig(F)$ is supported on at least $d-\rk(F)+2$ distinct points then $\Eig(F)\cap \W(F)\neq \emptyset$.  
\end{proposition}
\begin{proof}
    Since $\rk(F)>(d+1)/2$, by \cite[Theorem 3.5, item (2)]{carlini2017waring} the forbidden locus of $F$ is the zero locus of the lower degree (non-square-free) generator $G$ of $\mathrm{Ann}(F)$, i.e., $\mathcal{F}(F)=V(G)$. The result follows by noticing that  $\deg(G)=d-\rk(F)+2<d$, so $G$ is supported in at most $d-\rk(F)+1$ points and $\Eig(F)$ is supported in at least $d-\rk(F)+2$ points.
\end{proof}

{
Let $(d+1)/2<r<d$. We remark that a binary form $F\in \sym^d(\CC^2)$ of rank $r$ can always be expressed as $F=L^{d-1}M+\sum_{i=1}^{d-r}L_i^d$, for linear forms $L,M,L_i$, where the binary form $\sum_{i=1}^{d-r}L_i^d $ is of rank $d-r$, see for instance \cite[Proposition 4.3]{buczynski2018locus}.
Hence, if we have a concrete expression for an $F\in \sym^d(\CC^2)$ of rank strictly higher than $(d+1)/2$ then it is possible to explicitly check whether all the roots of $\Eig(F)$ lie in $\mathcal F(F)$ by direct computation: indeed, it is enough to have a generator of minimal degree of the apolar ideal and compare its roots with the eigenscheme. To give an idea, in the following example we show how $\Eig(F)\cap \W(F)\neq \emptyset$ when $F=x^{d-1}y+L^d$ for some linear form $L\in \sym^1(\CC^2)$.
\begin{example}\label{example: rango d-1 waring e eigen si intersecano sempre}
Let $F=x^{d-1}y+(ax+by)^d$ be of rank $d-1$, so in particular $b\neq 0$. Since a generator of minimal degree of ${\rm Ann}(F)$ is $y^2(ax+by)^\perp$, the forbidden locus of $F$ is $\mathcal{F}(F)=V(y^2(bx-ay))$, that is $\mathcal{F}(F)=\{ [1:0],[\frac{a}{b}:1] \}$.
We can easily compute the eigenscheme of $F$ as the zero locus of the polynomial
\begin{align*}
 D&=   \det\begin{bmatrix}
    (d-1)x^{d-1}y+a(ax+by)^{d-1} & x^{d-1}+b(ax+by)^{d-1}\\
    x & y
\end{bmatrix}\\
&=(ax+by)^{d-1}(ay-bx)+x^{d-2}((d-1)y^2-x^2).
\end{align*}
One explicitly checks that
\begin{align*}
 & D([1:0])=0 \iff a^{d-1}=-\frac{1}{b},\\
 &D([\frac{a}{b}:1])=0 \iff a^{d-2}(b^2(d-1)-a^2)=0\iff a=0 \text{ or }a=\pm b\sqrt{d-1}.
\end{align*} 
Then, if $a\notin \{0, \pm b\sqrt{d-1}\} \cup \{\left(-\frac{1}{b}\right)^{\frac{1}{d-1}}\omega \mid \omega^{d-1} = 1\}$  then $\mathcal{F}(F)\cap \Eig(F)=\emptyset$. On the other hand, even if $a\in \{0, \pm b\sqrt{d-1}\} \cup \{\left(-\frac{1}{b}\right)^{\frac{1}{d-1}}\omega \mid \omega^{d-1} = 1\}$ then only one of the two elements of $\mathcal{F}(F)$ can correspond to a root of $\Eig(F)$, meaning that in any case we have $\Eig(F)\cap \W(F)\neq \emptyset$ since $\Eig(F)\not\subseteq \mathcal F(F)$.
\end{example}
}

\subsection{Monomials}
The case of \Cref{cor: osculante} was easy to handle since the form of the apolar ideal was clear. This can be generalized to the case of monomials with $n\geq 2$ variables, since the apolar ideal and the corresponding Waring locus are well known.

 Let $F=x_0^{d_0}\cdots x_n^{d_n}$ with $1 \leq d_0\leq \cdots \leq d_n$ for some $n\geq 2$ and let $m=\max\{i \ \vert \ d_i=d_0 \}$. By \cite[Theorem 3.3]{carlini2017waring} the forbidden locus of $F$ is $\mathcal{F}(F)=V(x_0\cdots x_m)$. The following result shows that for a monomial $F$ both the intersections $\mathcal{F}(F)\cap \mathrm{Eig}(F)$ and  $\mathcal{W}(F)\cap \mathrm{Eig}(F)$ are both nonempty.

\begin{theorem}\label{prop: monomial case}
 Let $n\geq 2$, $F=x_0^{d_0}\cdots x_n^{d_n}$ with $1\leq d_0\leq \cdots \leq d_n$, and $m=\max \{ i \, \vert \, d_0=d_i \}$. If $d_0\geq 2$ then $\mathcal{F}(F)\subsetneq  \mathrm{Eig}(F)$, otherwise  $\Eig(F)\cap\mathcal F(F)=\bigcup_{i=0}^{m} \bigcup_{j\neq i} H_{i,j}$ where $H_{i,j}=V(x_i,x_j)$. In particular, in all cases we have $\Eig(F)\cap \W(F)\neq \emptyset$.

\end{theorem}
\begin{proof}
Denote by $I$ the ideal defined in \cref{eq: defining ideal eigenscheme}. Assume that $d_0\geq 2$. Then the ideal becomes
\begin{align*}
I&=\left(\frac{x_0^{d_0}\cdots x_n^{d_n}}{x_ix_j}(\sqrt{d_i}x_j-\sqrt{d_j}x_i)(\sqrt{d_i}x_j+\sqrt{d_j}x_i) \right)_{i<j}, 
\end{align*}
from which it is obvious that $\mathcal{F}(F)\subset \Eig(F)$. The inclusion is strict since the two ideals are generated in different degrees.

We now assume $d_0=1$, so $F=x_0\cdots x_m x_{m+1}^{d_{m+1}}\cdots x_n^{d_n}$. In this case $I$ is more involved. The 2-minors in \cref{eq: defining ideal eigenscheme}, considering columns $i<j$, are the following:
\begin{align}\label{eq: ideali 2-minori}
\begin{cases}
    \prod\limits_{\substack{\ell \leq m\\ \ell\neq i,j}}x_\ell \prod\limits_{\ell \geq m+1} x_\ell^{d_\ell}(x_j-x_i)(x_j+x_i) & \text{ if } i<j\leq m,\\
    \prod\limits_{\substack{\ell \neq i\\ \ell \leq m}} x_\ell \prod\limits_{\substack{\ell \neq j\\ \ell \geq m+1}}x_\ell^{d_\ell} x_j^{d_j-1}(x_j-d_jx_i) & \text{ if } i\leq m< j, \\
  \prod\limits_{ \ell \leq m} x_\ell \prod\limits_{\substack{\ell \neq i,j\\ \ell \geq m+1}}x_\ell^{d_\ell}  x_i^{d_i-1}x_j^{d_j-1}(\sqrt{d_i}x_j-\sqrt{d_j}x_i) (\sqrt{d_i}x_j+\sqrt{d_j}x_i) & \text{ if }m<i<j.
\end{cases}
\end{align}

\bigskip

{We remark that $ \mathcal F(F)=V(x_0\cdots x_m)$, and if we denote by $H_i=V(x_i)$ then $\mathcal F(F)=\bigcup_{i=0}^mH_i $. Hence, a point $p=[p_0:\dots:p_n]\in \mathcal F(F)$ has $p_i=0$ for some $ 0\leq i\leq m$ and $\Eig(F)\cap \mathcal F(F)=\bigcup_{i=0}^m (H_i\cap \Eig(F))$. To prove the result, we show that if $p\in \Eig(F)$, then there is another $p_j=0$, with $j\neq i$. We show the result when $p_0=0$, and the other cases for $i\in\{1,\dots,m\}$ follows by the same argument.

Notice that if $p_0 = 0$, the third line of \cref{eq: ideali 2-minori} is trivial, as it factors through $x_0$. In the first and second lines of \cref{eq: ideali 2-minori}, the evaluation at $p$ is a monomial where all $p_1,\ldots,p_n$ appear. 
Thus if we want it to vanish then there must be $j\ne0$ such that $p_j=0$. Therefore the vanishing of the ideal intersected with $\mathcal F(F)$ is
\begin{equation*}
\Eig(F)\cap\mathcal F(F)=\bigcup_{i=0}^{m} \bigcup_{j\neq i} H_{i,j}. \qedhere
\end{equation*}
}\end{proof}

\section{The critical Waring variety}\label{section: critical W locus}

In this section we focus on the variety of forms $F$ such that $\mathcal W(F)\cap \mathrm
{Eig}(F)\neq \emptyset$.

\begin{definition}\label{def: critical waring locus}
    Fix integers $n,d\geq 2$ {and $r \geq 1$}. The \emph{critical Waring variety of rank-$r$ forms} or the \emph{$r$-critical Waring variety} is
    \[
        \mathrm{WE}_{n,r}^d:=\overline{\{F\in \sigma_r^\circ(V_n^d)\mid \mathcal W(F)\cap \mathrm{Eig}(F)\neq \emptyset\}}\subset \PP^{\binom{n+d}{d}-1}=\PP(\sym^d(\CC^{n+1})),
    \]
    where the closure is taken in the Zariski topology.

\end{definition}

\begin{remark}\label{rmk: invariance}
In this section, when studying the elements in $\mathrm{WE}_{n,r}^d$ we will often restrict to assume that $F$ has $[x_0]\in \mathcal W(F)\cap \mathrm{Eig}(F)$. We remark that there is no loss of generality in assuming so, due to the equivariance of the eigenscheme and the Waring locus under the natural action of $\mathrm{O}(n+1)$ in $\sym^{d}(\CC^{n+1})$ induced by the action of the same group in $\sym^1(\CC^{n+1})\cong \CC^{n+1}$. 

More precisely, if $A\in \mathrm O(n+1)$, then $\Eig(A\cdot F)=A\cdot \Eig(F)$, see for instance \cite{Ott22}, and also $\mathcal W(A\cdot F)=A\cdot \mathcal W(F)$. In fact, it is classically known that the latter holds true for $A\in \mathrm{GL}(n+1)$. 

In practice, it means that we can fix a linear form $L$ and study the locus $X_L$ of forms $F\in \sym^d(\CC^{n+1})$ such that $[L]\in \mathcal W(F)\cap\mathrm{Eig}(F)$. It follows that $A\cdot [L]\in \mathcal W(A\cdot F)\cap\mathrm{Eig}(A\cdot F)$. Thus, for a non-isotropic $L$, 
$$
\mathrm{WE}_{n,r}^d=\overline{ \mathrm{O}(n+1)\cdot X_L}. 
$$

\end{remark}
We now aim to obtain a more uniform description of $\mathrm{WE}_{n,r}^d$ utilising the aforementioned symmetry reduction. 

More precisely, a degree $d$-form $F\in \sym^d (\CC^{n+1})$ of rank at most $r$ can always be written, up to  $\mathrm{O}(n+1)$-action and rescaling, as $F=x_0^d+ \sum_{i=1}^{r-1} (\alpha_{i,0}x_0+\cdots +\alpha_{i,n}x_n)^d\in\sym^d(\CC^{n+1})$. Using the the so-called \emph{$r$-abstract secant variety}
$$
\mathrm{Ab}\sigma_r(V_n^d)=\{ (G, (p_1,\dots,p_r)) \, |\, G\in \langle p_1,\dots, p_r \rangle \} \subset \PP(\sym^d(\CC^{n+1})) \times  (\PP^n)^r
$$
we can look at the parameter space $\PP(\CC[\alpha_{i,j}])\cong \PP^{(n+1)(r-1)-1}$ of the $(r-1)$-abstract secant variety containing $[F]$ and give equations characterizing $X_{x_0}$ and compute its dimension and degree.

\begin{theorem}\label{thm:equationsx0}
     Let $F=x_0^d+\sum_{i=1}^{r-1} (\alpha_{i,0}x_0+\cdots +\alpha_{i,n}x_n)^d\in\sym^d(\CC^{n+1})$ with $r\leq \binom{n+d}{d}/(n+1)$ . The locus in $\PP^{(n+1)(r-1)-1}$ for which $x_0$ is an eigenvector is a degree $d^n$ variety of codimension $n$ with equations 
     \begin{align}\label{eq: equations gj}
     g_j=\sum_{i=1}^{r-1}\alpha_{i,0}^{d-1}\alpha_{i,j} \text{ for }j=1,\dots, n
     \end{align}
     in $\PP^{(n+1)(r-1)-1}$. In particular, when $n=1$, this locus is a degree $d$ hypersurface. Moreover, the locus for which $x_0$ has singular value equal to $1$ is a degree $d^{n+1}$ variety of codimension $n+1$ with additional equation $\sum_{i=1}^{r-1} \alpha_{i,0}^d$.
\end{theorem}

\begin{proof}
    By \Cref{lemma:apolar vs eig} we have that $x_0^{d-1}(\partial)\circ (\lambda x_0^d+\sum_{i=1}^{r-1} (\alpha_{i,0}x_0+\cdots +\alpha_{i,n}x_n)^d)=0$,  therefore $$
    \begin{cases}
        \sum_{i=1}^{r-1}\alpha_{i,0}^d+\lambda=0,\\
        \sum_{i=1}^{r-1} \alpha_{i,0}^{d-1}\alpha_{i,1}=0,\\
        \ \ \ \ \  \vdots\\
        \sum_{i=1}^{r-1} \alpha_{i,0}^{d-1}\alpha_{i,n}=0.
    \end{cases}
    $$
    Denote by $g_i$ the equation corresponding to the system's $(i+1)$-th row. Notice that any solution of $I=(g_1,\dots, g_n)$ gives also a solution of $g_0$ for an appropriate $\lambda$, which corresponds to $1-\mu$, where $\mu$ is the singular value of $x_0$. 
    
    We first analyse the variety defined by $I$, i.e., the locus such that $[x_0]\in\Eig(F)\cap \mathcal W(F)$. Let $
    J=\begin{bmatrix}
        \frac{\partial g_1}{\partial \alpha_{i,j}}&\cdots& \frac{\partial g_n}{\partial \alpha_{i,j}} 
    \end{bmatrix}
    $ be the Jacobian matrix of the polynomials $g_1,\dots,g_n$, then $\codim V(I)=\mathrm{rank} (J)$ on a generic point. Consider the submatrix of $J$ corresponding to the derivatives with respect to $\alpha_{i,j}$ for a fixed $i$, with $j$ ranging from $1$ to $n$:
    $$
    \begin{bmatrix}
        \alpha_{i,0}^{d-1}\\
        &\alpha_{i,0}^{d-1}\\
        &&\ddots\\
        &&&\alpha_{i,0}^{d-1}
    \end{bmatrix}.
    $$
  It has rank $n$ in the open subset $\alpha_{i,0}\neq0$. Moreover, it is easy to see that $V(I)\not\subset \{\alpha_{i,0}=0\mid i=1,\dots,r-1\}$, thus $\codim V(I)=n$.

    Consider now the additional equation $g_0'=\sum_{i=1}^{r-1}\alpha_{i,0}^d$ corresponding to when $x_0$ has singular value $\mu=1$ and let $I'=(g_0',g_1,\dots,g_n)$. As before we consider a submatrix of the Jacobian of $g_0',g_1,\dots,g_n$  corresponding to the derivatives of $\alpha_{i,j}$ with fixed $i$ and $j$ ranging from $0$ to $n$. We have
$$
    \begin{bmatrix}
        d\alpha_{i,0}^{d-1}&(d-1)\alpha_{i,0}^{d-2}\alpha_{i,1}\\
        &\alpha_{i,0}^{d-1}&(d-1)\alpha_{i,0}^{d-2}\alpha_{i,2}\\
        &&\alpha_{i,0}^{d-1}\\
        &&&\ddots\\
        
        &&&\alpha_{i,0}^{d-1}&(d-1)\alpha_{i,0}^{d-2}\alpha_{i,n}\\
        &&&&\alpha_{i,0}^{d-1}
    \end{bmatrix},
$$
which has rank $(n+1)$, and again $V(I')\not\subset \{\alpha_{i,0}=0\mid i=0,\dots,n\}$, thus $V(I')$ has codimension $n+1$.
\end{proof}

For $i=1,\dots, n$ denote by $H_i\subset \PP^{\binom{n+d}{d}-1}$ the hyperplane defined by the vanishing of the apolar action against the monomial $x_0^{d-1}x_i$, i.e., $H_i = \{F\in {\rm Sym}^d(\mathbb{C}^{n+1}) \mid \langle x_0^{d-1}x_i, F \rangle = 0\}$. 
The conditions obtained in \Cref{thm:equationsx0} define a locus in $\sym^d(\CC^{n+1})$ and its closure can be understood as the $x_0^d$ joined with $\sigma_r^\circ(V_n^d)$ intersected with the hyperplanes $H_i$.

\begin{corollary}\label{cor: join}
    Let $F=x_0^d+G$, with $[G]\in \sigma_{r-1}^\circ(V_n^d)$. Then $x_0^d$ is an eigenvector of $F$ if and only if $H_i (G)=0 $, for $i=1,\dots, n$.
\end{corollary} 

This allows us to show that $\mathrm{WE}_{n,r}^d$ contains rank $r$ forms in many cases.

\begin{proposition}\label{prop:rankrbinary}
    The variety defined by the equations $g_1,\dots,g_n$ of \cref{eq: equations gj} contains a polynomial of Waring rank $r\leq \frac{d+1}{2}$.
\end{proposition}
\begin{proof}
  We will show the result for $V_1^d\subseteq V_n^d$. It is sufficient to prove that $$\{[1:0],[\alpha_{1,0}:\alpha_{1,1}],\dots, [\alpha_{r-1,0}:\alpha_{r-1,1}]\}\subset\PP^1$$ consists of $r$ distinct points. Indeed, in such case we have that $F=x_0^d+\sum_{i=1}^{r-1} (\alpha_{i,0}x_0+\alpha_{i,1}x_1)^d$ is annihilated by $\partial_1\prod_{i=1}^{r-1}(\alpha_{i,1}\partial_0-\alpha_{i,0}\partial_1)$, which is square-free, and thus $\mathrm{rk}(F)=r$. It is easy to verify that the equation $\sum_{i=1}^{r-1}\alpha_{i,0}^{d-1}\alpha_{i,1}=0$ has a solution satisfying such condition, for instance letting $\alpha_{i,0}=1$, we obtain a linear equation $\sum_{i=1}^{r-1}\alpha_{i,1}=0$, and for instance $[\alpha_{i,0}:\alpha_{i,1}]=[1:\xi^i]$ provides a solution with the required property, where $\xi$ is a $(r-1)$-fundamental root of unity.
\end{proof}

\begin{proposition}\label{prop:rk in WE}
    Let $m_{n-1}$ be the maximum Waring rank in $\sym^d(\CC^n)$. Then $\mathrm{WE}_{n,r}^d$ contains a polynomial of Waring rank $r$ for all $r\leq \max\{\frac{d+1}{2}, m_{n-1}+1\}$. 
\end{proposition}
\begin{proof}
  When $\max\{\frac{d+1}{2}, m_{n-1}+1\}=\frac{d+1}{2}$, the result follows from \Cref{prop:rankrbinary}. 
  Otherwise, let $G\in \sym^d(\CC^n)$ in the variables $x_1,\dots,x_n$ with $\rk(F)=r-1\leq m_{n-1}$. By \cite[Theorem 1.1]{CCC15}, it follows that $F=x_0^d+G$ has rank $r$ and $[x_0]\in\W(F)$. Moreover, as the monomials $x_0^{d-1}x_j$ have coefficient equal zero for $j\neq 0$, it follows that $[x_0]\in\Eig(F)$, thus $[F]\in \mathrm{WE}_{n,r}^d$.
\end{proof}

This allows us to give another description of $\mathrm{WE}_{n,r}^d$ as orbit closure of the action of $\mathrm{O}(n+1)$.

\begin{remark}\label{rmk: orbitclosure}
By \Cref{rmk: invariance} and \Cref{cor: join}, we have

$$\mathrm{WE}_{n,r}^d=\overline{\bigcup_{g\in \mathrm{O}(n+1)} g\cdot\left( \mathcal{J}( [x_0^d],\sigma_{r-1}^\circ(V^d_n)\cap (H_1\cap \cdots \cap H_n)\right))}.$$
We notice that this description is not so simple to work with as it involves intersecting hyperplanes with a dense locus in $\sigma_{r-1}(V_n^d)$. In the case of binary forms for $r \leq \frac{d+1}{2}$, we will provide sufficient conditions ensuring that
\begin{equation}\label{eq:closure}
\overline{\sigma_{r-1}^\circ(V_1^d)\cap H_1}=\sigma_{r-1}(V_1^d)\cap H_1.\end{equation}
To guarantee that \cref{eq:closure} holds, it is sufficient to observe that $\sigma_{r-1}(V_1^d)\cap H_1$ is irreducible (see \Cref{thm: irriducibilita'}) and that there exists an element $[F]\in \sigma_{r-1}^\circ(V_1^d)\cap H_1$ with $\rk(F)=r-1$ (see \Cref{prop:rankrbinary}). Indeed, under such assumptions, as the set $\{F\mid \rk (F)=r-1\}$ is dense in $\sigma_{r-1}(V_1^d)$, it remains dense in the induced topology in $\sigma_{r-1}(V_1^d)\cap H_1$. Consequently, for binary forms with $r \leq \frac{d+1}{2}$, we have: 
\begin{equation}\label{eq: WE orbita intersezione chiuso}
\mathrm{WE}_{1,r}^d=\overline{\bigcup_{g\in \mathrm{O}(2)} g\cdot\left( \mathcal{J}( [x_0^d],\sigma_{r-1}(V^d_1)\cap H_1\right))}.
\end{equation}

\end{remark}

\begin{corollary}\label{cor:equations2}
    Let $g\in \mathrm O(n+1)$, and let $L=g\cdot x_0$. Then $[F]\in \mathcal J([L^d],\sigma_{r-1}^\circ(V_n^d))$ has $L$ as an eigenvector if and only if $F\in{\Tilde{H_i}}$, where $\Tilde{H_i}$ is the hyperplane defined by the vanishing of the apolar action against $(g\cdot x_0)^{d-1}(g\cdot x_i)$.
\end{corollary}

Let us consider $\PP^{\binom{n+d}{d}-1}=\PP(\sym^{d}(\CC^{n+1}))$ with coordinates $Y_{\alpha}$, where $\alpha=(\alpha_0,\dots,\alpha_n)$ is such that $|\alpha|=\alpha_0+\cdots+\alpha_n=d $, and $Y_\alpha$ is the coefficient of the monomial $x^\alpha=x_0^{\alpha_0}\cdots x_n^{\alpha_n}$. 
Notice that in this choice of coordinates the hyperplane $H_i$, defined by the vanishing of the apolar action against $x_0^{d-1}x_i$, is the hyperplane with equation $Y_{(d-1,0\dots,0,1,0,\dots,0)}=0$. We make the equation of the hyperplanes $\Tilde{H_i}$ for $i=1,\dots,n$ more explicit. 

Similar to the notation for monomials, for a vector $v\in \CC^{n+1}$ and $\alpha\in\NN^{n+1}$, we denote by $v^\alpha=\prod_{i=0}^n v_i^{\alpha_i}\in\CC$.
\begin{corollary}
    Let $g\in \mathrm O(n+1)$ with a matrix representation $g=(g_{i,j})_{i,j=0}^n$ and $L=g\cdot x_0$. Then the $\Tilde{H_i}$'s in \Cref{cor:equations2} have equation $$
    \sum_{|\alpha|=d}\sum_{\ell=0}^n\binom{d-1}{\alpha_0,\dots,\alpha_\ell-1,\dots,\alpha_n}g_0^{\alpha-e_\ell}g_{\ell i}Y_\alpha=0.
    $$
    Where $g_k$ represents the $k$-th column vector of the matrix $g$ and $e_i$ is the i-th vector in the standard basis of $\CC^{n+1}$.

\end{corollary}
\begin{proof} 
We compute  the image $(g\cdot x_0)^{d-1}(g\cdot x_i)$ explicitly.  
    \begin{align*}
        (g\cdot x_0)^{d-1}(g\cdot x_i)&=\left(\sum_{\ell=0}^n g_{\ell,0}x_\ell\right)^{d-1}\left(\sum_{\ell=0}^n g_{\ell ,i}x_i\right)\\
        &= \left(\sum_{|\alpha|=d-1}\binom{d-1}{\alpha_0,\dots,\alpha_n}g_0^\alpha x^\alpha\right)\left(\sum_{\ell=0}^n g_{\ell, i}x_\ell\right)\\
        &=\sum_{|\alpha|=d-1}\sum_{\ell=0}\binom{d-1}{\alpha_0,\dots,\alpha_n}g_0^\alpha g_{\ell ,i}x^{\alpha+e_i}\\
        &=\sum_{|\alpha|=d}\sum_{\ell=0}\binom{d-1}{\alpha_0,\dots,\alpha_\ell - 1,\dots,\alpha_n}g_0^{\alpha-e_i} g_{\ell ,i}x^{\alpha}.
    \end{align*}
    Therefore, the hyperplane $\tilde{H}_i$ has equation \[
    \sum_{|\alpha|=d}\sum_{\ell=0}^n\binom{d-1}{\alpha_0,\dots,\alpha_\ell-1,\dots,\alpha_n}g_0^{\alpha-e_\ell}g_{\ell, i}Y_\alpha=0.\qedhere
    \]
\end{proof}

\subsection{{Dimension and irreducibility of critical Waring variety}}

We now aim to understand the dimension and irreducibility of $\mathrm{WE}_{n,r}^d$ by studying the orbit closure of $\mathcal{J}( [x_0^d],\sigma_{r-1}(V^d_n))\cap (H_1\cap \cdots \cap H_n))$ with respect to the action of $\mathrm{O}(n+1)$, see \Cref{rmk: orbitclosure}.

\begin{theorem}\label{proposition: dim WE rnc }
    Let $2\leq r\leq \frac{d+1}{2}$, then $\dim(\mathrm{WE}_{1,r}^d)=2(r-1)$. In particular, $\mathrm{WE}_{1,r}^d$ is codimension one subvariety of $\sigma_r(V_1^d)$.
\end{theorem}
\begin{proof}
    {The strategy is: we express $\mathrm{WE}^d_{1,r}$ as the image of a map 
    $
        \psi:\mathrm{O}(2)\times Y\to \sigma_r(V_1^d),\ \psi(g,y)=g\cdot y,
    $
    then, we identify a specific element for which the fiber $\psi^{-1}(z)$ is zero-dimensional and conclude by semi-continuity.}
    
    Since $\sigma_{r-1}(V_1^d)\subsetneq\mathcal J([x_0^d], \sigma_{r-1}(V_1^d))$, then $\dim \mathcal J([x_0^d], \sigma_{r-1}(V_1^d)) = \dim \sigma_{r-1}(V_1^d)+1 = 2(r-1).$ Let $Y=H_1\cap \mathcal J([x_0^d], \sigma_{r-1}(V_1^d))$ We have $\dim(Y)=2(r-1)-1$.
    As mentioned, $\mathrm{WE}_{1,r}^d=\mathrm{im}(\psi)$. Thus, $$\dim \mathrm{WE}_{n,r}^d=\dim Y+\dim \mathrm{O}(2)-\dim \psi^{-1}(z),$$ where $z$ is a generic element in the image of $\psi$ and $\psi^{-1}(z)=\{(g,y)\in \mathrm{O}(2)\times Y\mid g\cdot y=z\}$. Let $\pi_2:\mathrm{O}(2)\times Y\to Y$ be the projection onto the second factor, so $$\dim \psi^{-1}(z)=\dim \pi_2(\psi^{-1}(z))+\dim \pi_2^{-1}(y)$$ for a generic $y\in \pi_2(\psi^{-1}(z))$. Notice $\pi_2(\psi^{-1}(z))=\{y\mid g\cdot y=z \text{ for some $g\in\mathrm{O}(2)$}\}$ and $\pi_2^{-1}(y)=\{(g,y)\mid g\cdot y=z\}$. We will show that both are zero-dimensional for a specific choice of $z$, and by semi-continuity, also the generic fibre of $\psi$ is zero-dimensional. 
    
    We fix $z=x_0^d+\sum_{i=1}^{r-1}(x_0+\xi^ix_1)^d\in Y\subset \mathrm{im}(\psi)$, for $\xi$ a fundamental $(r-1)$-root of unity, where the belonging follows from the proof of \Cref{prop:rankrbinary}.

    Since $y\in Y$, $y=x_0^d+\sum_{i=1}^{r-1}L_i^d$, and notice that $y\in \pi_2(\psi^{-1}(z))$ if and only if $y\in Y\cap \mathrm{O}(2)\cdot z$. By assumption, both $y$ and $z$ have rank $r\leq \frac{d+1}{2}$, both are identifiable, therefore $g\cdot\{ x_0,(x_0+\xi^ix_1)\}_{i=1}^{r-1}=\{\mu_0 x_0,\mu_iL_i\}_{i=1}^{r-1}$, where $\mu_0^d=\mu_i^d=1$. It follows that either $g\cdot x_0=\mu_0 x_0$ or $g\cdot (x_0+\xi^ix_1)=\mu_0x_0$ for some $i$. 

    If $g\cdot x_0=\mu_0 x_0$, then $g=\begin{bmatrix}
        \mu_0 &0\\
        0&\eta
    \end{bmatrix}$, with $\mu_0^2=\eta^2=1$, and we obtain finitely many such $y$'s. On the other hand, if $g\cdot \frac{1}{\|(1,\xi^i)\|}(x_0+\xi^ix_1)=\mu_0 x_0$, then 
    $$
    g=\frac{1}{\|(1,\xi^i)\|}\begin{bmatrix}
       \mu_0-\xi^{2i}&\xi^i\\
       -\xi^i&\mu_0-\xi^{2i}
    \end{bmatrix} \quad \text{ or } \quad g=\frac{1}{\|(1,\xi^i)\|}\begin{bmatrix}
       \mu_0-\xi^{2i}&\xi^i\\
       \xi^i&-\mu_0+\xi^{2i}
    \end{bmatrix}.
    $$
   It follows that for each $i=1,\dots, r-1$, there are at most finitely many such $y$'s in $Y\cap \mathrm{O}(2)\cdot z$. Finally, we notice that if $\|(1,\xi^i)\|=0$, i.e., it is isotropic, then there is no $g\in \mathrm{O}(2)$ such that $g\cdot x_0=\lambda (x_0+\xi^ix_1)$ for any $\lambda\in \CC$, and the finiteness argument still holds. Thus $\dim \pi_2(\psi^{-1}(z))=0$.

    We now focus on $\pi_2^{-1}(y)$: let $y=x_0^d+\sum_{i=1}^{r-1}L_i^d\in \pi_2(\psi^{-1}(z))$, since $g\cdot y=z$, and $y$ is identifiable, we have $g\cdot \{x_0,L_i\}_{i=1}^{r-1}=\{\mu_0x_0,\mu_i(x_1+\xi^ix_1)\}_{i=1}^{r-1}$, with $\mu_0^d=\mu_i^d=1$. Similarly, either $g\cdot x_0=\mu_0x_0$ or $g\cdot x_0=\frac{\mu_i}{\|(1,\xi^i)\|}(x_0+\xi^ix_1)$. For the first, again it follows that $g$ is diagonal with entries $\pm 1$. For the second, it follows that
  $$
 g=\frac{1}{\|(1,\xi^i)\|}\begin{bmatrix}
        1&\xi^i\\
        -\xi^i&1
    \end{bmatrix} \quad \text{ or }  \quad g=\frac{1}{\|(1,\xi^i)\|}\begin{bmatrix}
        1&\xi^i\\
        \xi^i&-1
    \end{bmatrix} ,$$  
    and there are finally many possibilities. Again, if $(1,\xi^i)$ is isotropic, there are no such matrices, and finiteness still holds.

   Hence $\psi^{-1}(z)$ is zero-dimensional, so also the generic fibre is zero-dimensional, and we conclude that $\dim(\mathrm{WE}_{1,r}^d)=\dim Y+\dim \mathrm{O}(2)=2(r-1)$.
    \end{proof}

   We now want to prove that the generic element in $\mathrm{WE}_{1,r}^d$ has rank $r$. To this end, let us first show that
   $\mathrm{WE}_{1,r}^d$ is irreducible.

\begin{lemma}\label{thm: irriducibilita'}
    Assume $2\leq r\leq \frac{d+1}{2}$, then $\sigma_r(V_1^d)\cap H_1$ is irreducible.
\end{lemma}
\begin{proof}
We recall that the equations of secant varieties of rational normal curves are given by the $(r+1)$-minors of catalecticant matrices (cf. for instance \cite[Section 0.1]{IarrobinoKanev}). Moreover, it is well known that one can reduce to consider only the case of maximal minors of the catalecticant matrix of size $(d-r+1)\times(r+1)$, see \cite[Lemma 2.3]{Peskine}. To prove the irreducibility of the  hyperplane section, we first notice that by \cite[Proposition 4.2]{EisenbudLinearSec} the space of catalecticant matrices $$\begin{bmatrix}
     a_0& a_1 & a_2&\cdots & a_{r}\\
     a_1 & a_2 & a_3& \dots & a_{r+1}\\
     \vdots & & & &\vdots \\
     a_{d-r}& a_{d-r+1}& \dots &\dots & a_d
 \end{bmatrix}$$ are $1$-generic. The linear section given by $H_1$ corresponds to considering $a_1=0$. We conclude by \cite[Theorem 1-(ii)]{eisenbud} that after setting $a_1=0$, the ideal of maximal minors  is still prime.
\end{proof}

\begin{theorem}\label{thm: irriducibilita'WE}
    Let $1\leq r\leq \frac{d+1}{2}$. The variety $\mathrm{WE}_{1,r}^d$ is irreducible.
\end{theorem}
\begin{proof}
     If $r\geq 3$, notice that, $Y=\mathcal J([x_0^d],\sigma_{r-1}(V_1^d))\cap H_1 =\mathcal J([x_0^d],\sigma_{r-1}(V_1^d)\cap H_1)${: indeed, $[x_0^d] \in H_1$, thus a line $\langle [x_0^d],[F]\rangle \subset H_1$ if and only if $[F] \in H_1$}. Since the join of irreducible varieties is still irreducible, $Y$ is irreducible by \Cref{thm: irriducibilita'}. Moreover, $\mathrm{WE}_{1,r}^d$ is the orbit closure of the action of $\mathrm{O}(2)$, an irreducible group, on $Y$, so it is also irreducible.

     For the case $r=2$, notice that if $F=x_0^d+(a_0x_0+a_1x_1)^d$ and has $x_0^d$ as eigenvector, then by \Cref{prop:essential_var}, $a_0a_1=0$, so, assuming $F$ has rank $2$, we have $F\in\{x_0^d+\lambda x_1^d\mid \lambda\in \CC^\ast\}$. The orbit closure of these polynomials with respect to $\mathrm{O}(2)$ is the variety of odeco polynomials, which is irreducible. Lastly, for $r=1$, $\mathrm{WE}_{1,1}^d=V_{1}^d$.
\end{proof}
Combining  \Cref{prop:rankrbinary} and \Cref{thm: irriducibilita'WE} provides the following.
\begin{corollary}
Let $1\leq r\leq \frac{d+1}{2}$. The generic element of $\mathrm{WE}_{1,r}^d$ is a rank-$r$ tensor.
\end{corollary}

In the general case, with $n\geq 2$, we provide a lower bound for the dimension.

\begin{theorem}\label{teo:dimWEn}
    Let $1\leq r\leq n$, $d\geq 3$. Then $ \dim (\mathrm{WE}_{n,r+1}^d)\geq r(n+1)$. 
\end{theorem}

\begin{proof}

    Set $H=\cap_{i=1}^nH_i$. Let $Y=\mathcal J([x_0^d],\sigma_r^\circ(V_n^d)\cap H)$ and let $\psi:\mathrm{O}(n+1)\times Y\to \mathrm{WE}_{n,r+1}^d\subset \sigma_{r+1}(V_n^d)$, and $\pi_2:\mathrm{O}(n+1)\times Y\to Y$. Notice that under our numerical assumptions, by Alexander-Hirschowitz Theorem, $\sigma_r(V_n^d)$ is non-defective, thus we have $\dim \sigma_r(V^d_n) = r(n+1)-1$. Moreover since $\sigma_r(V_n^d)$ is not a cone,  $\dim \mathcal{J}([x_0^d],\sigma_r^\circ(V_n^d)) = r(n+1)$. Finally, since $\sigma_r^\circ(V_n^d)\cap H\neq \emptyset$, it follows $\dim Y \geq r(n+1)-n$. By using the classical theorem on the dimension of the generic fibre twice, following the same strategy as in \Cref{proposition: dim WE rnc }, we have that
    \begin{align*}
        \dim(\mathrm{WE}_{n,r+1}^d)& = \dim Y+\dim \mathrm{O}(n+1)-\dim \psi^{-1}(z)\\
        & = \dim Y+\dim \mathrm{O}(n+1)-\dim \pi_2(\psi^{-1}(z))-\dim \pi_2^{-1}(y),
    \end{align*}
    where $z$ is generic in $\mathrm{WE}_{n,r+1}^d$ and $y$ is generic in $\pi_2(\psi^{-1}(z))$.
    
    Since we want to employ the semi-continuity of the dimension of the fibre, we fix $z=\sum_{i=0}^rx_i^d$, then $\pi_2(\psi^{-1}(z))=\{y\in Y\mid g\cdot y=z \text{ for some $g\in \mathrm{O}(n+1)$}\} = (\mathrm{O}(n+1)\cdot z)\cap Y$.
    We first observe that $\mathrm{O}(n+1)\cdot z=\{y=\sum_{i=0}^r v_i^d\mid \{v_0,\dots,v_r\} \text{ is orthonormal}\}$, thus if $y\in(\mathrm{O}(n+1)\cdot z)\cap Y$, then also $v_i=x_0$ for some $i=0,\dots,r$, say $y=x_0^d+\sum_{i=1}^r v_i^d$. Moreover, since $\langle v_i,x_0\rangle=0$, it follows $v_i\in \CC[x_1,\dots,x_n]$. 
    Thus, $(\mathrm{O}(n+1)\cdot z)\cap Y = \mathrm{O}_{n,r}$, the space of $n \times r$-matrices whose columns are orthonormal. Note that this has dimension $\dim\mathrm O(n)-\dim \mathrm{O}(n-r)$: indeed, the map $\mathrm O(n) \rightarrow \mathrm O_{n,r}$ selecting the first $r$ columns is surjective and has generic fibre isomorphic to $\mathrm O(n-r)$.

    We have $\pi_2^{-1}(y)=\{g\in\mathrm O(n+1)\mid g\cdot y=z\}$, and notice $g\in \pi_2^{-1}(y)$ if and only if $g^{-1}\cdot z=y$, so it is equivalent to compute the dimension of $\{g\in \mathrm{O}(n+1)\mid g\cdot z=y\}$. Let $y=x_0^d+\sum_{i=1}^r v_i^d\in Y$, since $y$ is also in  $\mathrm O(n+1)\cdot z$, we have that $v_i\in \CC[x_1,\dots,x_n]$ and $\{x_0,v_1,\dots,v_r\}$ is orthonormal by the previous paragraph, in particular $y$ is identifiable. It follows that $g\cdot \{x_0,\dots,x_r\}=\{\mu_0x_0,\mu_1v_1,\dots, \mu_rv_r\}$, where $\mu_i^d=1$. Let $\{C_0,\dots,C_r\}=\{\mu_0x_0,\mu_1v_1,\dots, \mu_rv_r\}$, considering them as column vectors we have
    $$
    g=\begin{bmatrix}
        C_0&\cdots &C_r &B
    \end{bmatrix}
    $$
    where $B$ is a $(n+1)\times(n-r)$-matrix. Since $g^tg=I_{n+1}$, we have $$
        g^tg=\begin{bmatrix}
        I_{r+1}&C^tB\\
        B^tC&B^tB
    \end{bmatrix},
    $$
    thus $B^tC=0$ and $B^tB=I_{n-r}$, where $C=\begin{bmatrix}
        C_0&\cdots &C_r
    \end{bmatrix}$,  then the columns of $B$ are an orthonormal basis of $\langle x_0,v_1,\dots,v_r\rangle^\perp$,  so $\pi_2^{-1}(y)\cong \mathrm{O}(n-r)$. 
    
    Since $z$ is specific, we conclude by semi-continuity on the dimension of the fibre that 
    \begin{align*}
        \dim \mathrm{WE}_{n,r+1}^d&\geq \dim Y+\dim \mathrm{O}(n+1)-\dim \mathrm O(n-r) -(\dim (\mathrm O(n))-\dim (\mathrm O(n-r)))\\
        &\geq r(n+1)-n + {n+1\choose 2}-{n\choose 2}=r(n+1). \qedhere
    \end{align*}
\end{proof}

{The bound obtained in \Cref{teo:dimWEn} is sharp, since it is an equality for $n=1$, see \Cref{proposition: dim WE rnc }.}

\section{On eigenvectors of binary forms that increase Waring rank}\label{sec: grow}

The previous sections were devoted to understanding when eigenvectors can be used to decrease the Waring rank. A natural converse question is to understand when an eigenvector increases the Waring rank. 

For binary forms of subgeneric Waring rank, as a consequence of \Cref{proposition: dim WE rnc }, we know that generically the support of the eigenscheme is contained in the forbidden locus. An interesting question is to understand if there are special elements in the eigenscheme, and more generally in the forbidden locus, that increase the rank of the given form.
 \cite{SC10} shows that, in the case of real binary cubics, eigenvectors can actually increase the rank, even above the generic one. In this section we show that for generic $[F]\in\sigma_r(V_1^d)$, with $r\leq (d+1)/2$, all eigenvectors increase the Waring rank. The crucial observation is that in these cases, \textit{all} elements from the forbidden locus increases the rank. While this fact is known to specialists in the subgeneric setting, we include a proof for completeness.

\begin{lemma}\label{prop: subgenrkgrow}
    Let $d\geq 3$, $r<\frac{d+1}{2}$, and $F\in \sym^d(\CC^2)$ such that $\rk(F)=r$. For every $[L_0]\in \mathcal F(F)$ and every $\lambda\in\CC\setminus\{0\}$, we have $\rk(F+\lambda L_0^d)=r+1$.
\end{lemma}
\begin{proof}
    Let $F=\sum_{i=1}^rL_i^d$. Suppose $r< \frac{d+1}{2}$, notice that $L_0^\perp\cdots L_r^\perp  \in \mathrm{Ann}(F+\lambda L_0^d)=(g_1,g_2)$, with $\deg g_1=\rk(F+\lambda L_0^d),\ \deg g_2=d-\rk(F+\lambda L_0^d)+2$, and $g_1$ is square free when $\rk(F+\lambda L_0^d)\leq \frac{d+1}{2}$. Since $[L_0]\in\mathcal F(F)$, we have $\rk(F+\lambda L_0^d)\in\{r,r+1\}$. Assume by contradiction that $\rk(F+\lambda L_0^d)=r$. 
    In such case $\deg g_1=r,\ \deg g_2=d-r+2> r+1$, so $g_1\mid L_0^\perp\cdots L_r^\perp $. It follows that $F+L_0^d=\sum_{i=1}^rM_{i}^d$, with $[M_i]\in \{[L_0],\dots,[L_r]\}$, for all $i=1,\dots,r$. However, $[L_0^d],\dots, [L_r^d]\in V_1^d$, and since $r+1<d+1$, they are linearly independent, thus $F+\lambda L_0^d$ cannot be written as a linear combination of only $r$ of the $L_i^d$'s when $\lambda\neq 0$, a contradiction. 
\end{proof}

A similar result can be proved also when $d$ is odd and $r$ is the generic Waring rank is $\frac{d+1}{2}$.

\begin{lemma}\label{lemma: genrkgrow}
    Let $F\in\sym^d(\CC^{2})$, $d\geq 3$ odd and $\rk (F)=\frac{d+1}{2}$. Suppose $[L_0]\in \mathcal F(F)$, then there exists $\lambda$ such that $\rk(F+\lambda L_0^d)=\rk (F)+1$.
\end{lemma}
\begin{proof}
    Let $r_g=\rk(F)$ be the generic rank in $\sym^d(\CC^{2})$. We denote the locus of binary forms of rank $r$ as $S_r={\{[G]\in\PP(\sym^{d}(\CC^2))\mid \rk(G)=r\}}$, and $\overline{S_r}$ its Zariski closure. It follows from \cite[Proposition~4.3]{buczynski2018locus} that $ \overline{S_{r_g+1}}$ is a hypersurface, thus $\langle [F],[L_0^d]\rangle \cap \overline{S_{r_g+1}}\neq \emptyset$. 
    
    We recall that $\overline{S_{r_g+1}}=\left(\bigcup_{i=1}^{r_g-1}S_i\right)\bigcup\left(\bigcup_{j=r_g+1}^d S_j \right)$, cf. \cite[Theorem 1]{NTT}. {Since $[L_0]\in \mathcal{F}(F)$, we have} that $\rk(\mu F+\lambda L_0^d)\in\{1,r_g,r_g+1\}$ for $[\mu:\lambda]\in\PP^1$, thus $$\langle [F],[L_0^d]\rangle \cap \overline{S_{r_g+1}}=\langle [F],[L_0^d]\rangle\cap \left(S_1\cup S_{r_g+1}\right).$$ We recall that $S_1=V_1^d$, and since $d$ is odd, $d=2k-1$ for some $k$. By \cite[Theorem 4.1]{LS16} we have $\deg \overline{S_{r_g+1}}=2k(k-1)>2k-1=d=\deg V_1^d$, so $\langle [F],[L_0^d]\rangle\cap S_{r_g+1}\neq\emptyset$.
\end{proof}

\begin{corollary}\label{cor:rankgroweig}
    Let $2\leq r\leq\frac{d+1}{2}$, $[F]\in \sigma_r(V_1^d)$ generic. For every $[L_0]\in \Eig(F)$ there exists $\lambda\in\CC$ such that $\rk(F+\lambda L_0^d)=r+1$. In particular, if $r<\frac{d+1}{2}$ then it holds for all $\lambda\neq 0$.
\end{corollary}
\begin{proof}
    By \Cref{proposition: dim WE rnc }, if $[F]\in \sigma_r(V_1^d)$ is generic, then $\Eig(F)\cap \W(F)=\emptyset$. Thus, if $[L_0]\in \Eig(F)$, then $[L_0]\in\mathcal F(F)$. The claim follows from \Cref{prop: subgenrkgrow} and \Cref{lemma: genrkgrow}. \qedhere
\end{proof}

\begin{remark}
    We notice that the analogous of \Cref{lemma: genrkgrow} for even degree remains open and the technique utilised is not easily applicable. Indeed, in this case $\overline{S_{r_g+1}}$ is a codimension two variety, therefore there is no guarantee that the intersection of $\overline{S_{r_g+1}}$ and the line $\langle [F], [L_0^d]\rangle$ is non-empty. Furthermore, while looking for an analogous of \Cref{cor:rankgroweig} for $d$ even and $r=r_g=\lceil (d+1)/2\rceil$, we should relax the statement by considering only generic forms of generic rank. Indeed, the result cannot hold for \textit{any} form: if $F=L^{\frac{d}{2}}M^{\frac{d}{2}}$, with $L$ and $M$ generic linear forms, then $\rk(F)=r_g$ and, by \Cref{cor: osculante}, we know that $\Eig(F)\subset \W(F)$, thus the linear forms in the eigenscheme cannot be used to increase the rank. Since monomials $L^{\frac{d}{2}}M^{\frac{d}{2}}$ are not generic, there might be hope for generic forms of generic rank.      
\end{remark}

In the next example, we present a form for which only part of the eigenscheme can be used to increase the rank. That is the suprageneric binary forms already considered in \Cref{example: rango d-1 waring e eigen si intersecano sempre} and for which we employ the expression coming from \cite{buczynski2018locus}.

{
\begin{example}
Let us consider again \Cref{example: rango d-1 waring e eigen si intersecano sempre}, so let us take $F=x^{d-1}y+(ax+by)^d$ with $b \neq 0$. We know that $\rk(F)=d-1$ and that $\mathcal{F}(F)=\{[1:0],[\frac{a}{b}:1] \}$. For any $\lambda\neq 0$, adding $\lambda x^d$ leads to $x^{d-1}(\lambda x+y)+(ax+by)^d$, which still has rank $d-1$, while adding $-(ax+by)^d$ clearly increases the rank. As observed in \Cref{example: rango d-1 waring e eigen si intersecano sempre}, if $a \in \{0, \pm b\sqrt{d-1}\}$, then $ax+by \in {\rm Eig}(F)$. At the same time, all the other eigenvectors do not increase the rank.

\end{example}}

We conclude by emphasizing that 
the relation between the eigenscheme and the forbidden loci of forms of suprageneric rank still requires deeper investigation which we leave for future studies.

\bibliographystyle{alpha}
\bibliography{References.bib}

\Addresses

\end{document}